\documentclass[12pt, reqno]{amsart}
\usepackage{amsmath, amsthm, amscd, amsfonts, amssymb, graphicx, color}
\usepackage[bookmarksnumbered, colorlinks, plainpages]{hyperref}

\newtheorem{theorem}{Theorem}[section]

\theoremstyle{definition}
\newtheorem{definition}[theorem]{Definition}
\newtheorem{example}[theorem]{Example}

\theoremstyle{remark}
\newtheorem{remark}[theorem]{Remark}
\numberwithin{equation}{section}

\begin{document}
\setcounter{page}{1}

\title[On the complex valued metric-like spaces] {On the complex valued metric-like spaces}

\author[A. Hosseini and M. Mohammadzadeh Karizaki]{A. Hosseini and M. Mohammadzadeh Karizaki}
\address{ Amin Hosseini, Kashmar Higher Education Institute, Kashmar, Iran}
\email{\textcolor[rgb]{0.00,0.00,0.84}{a.hosseini@mshdiau.ac.ir}}

\address{M. Mohammadzadeh  Karizaki, Department of Computer Engineering, University of Torbat Heydarieh,
Torbat Heydarieh, Iran}
\email{\textcolor[rgb]{0.00,0.00,0.84}{m.mohammadzadeh@torbath.ac.ir}}

\subjclass[2010]{54A05, 46A19}

\keywords{Complex valued metric-like space, metric-like space, complex valued partial metric space, partial metric space, metric space}

\date{Received: xxxxxx; Revised: yyyyyy; Accepted: zzzzzz.
\newline \indent $^{*}$ Corresponding author}

\begin{abstract}
The main purpose of this paper is to study complex valued metric-like spaces as an extension of metric-like spaces, complex valued partial metric spaces, partial metric spaces, complex valued metric spaces and metric spaces. In this article, the concepts such as quasi-equal points, completely separate points, convergence of a sequence, Cauchy sequence, cluster points and complex diameter of a set are defined in a complex valued metric-like space. 
\end{abstract} \maketitle

\section{Introduction and preliminaries}
Distance is an important and fundamental notion in mathematics and there exist many generalizations of this concept in the literature (see \cite{De}). One of such generalizations is the partial metric which was introduced by Matthews (see \cite{Mat}). It differs from a metric in that points are allowed to have non-zero "self-distances" (i.e., $d(x,x) \geq 0$), and the triangle inequality is modified to account for positive self-distances. O'Neill \cite{O} extended Matthews definition to partial metrics with "negative distances". Before describing the material of this paper, let us recall some definitions and set the notations which we use in what follows.

\begin{definition} A mapping $\mathfrak{p} : X \times X \rightarrow \mathbb{R}^{+}$, where $X$ is a non-empty set, is said to be a
partial metric on $X$ if for any $x, y, z \in  X$, the following four conditions hold true:\\
(i) $x = y$ if and only if $\mathfrak{p}(x, x) = \mathfrak{p}(y, y) = \mathfrak{p}(x, y)$;\\
(ii) $\mathfrak{p}(x, x) \leq \mathfrak{p}(x, y)$;\\
(iii) $\mathfrak{p}(x, y) = \mathfrak{p}(y, x)$;\\
(iv) $\mathfrak{p}(x, z) \leq \mathfrak{p}(x, y) + \mathfrak{p}(y, z) - \mathfrak{p}(y, y)$.
\end{definition}
The pair $(X, \mathfrak{p})$ is then called a \emph{partial metric space.} A sequence $\{x_n\}$ in a partial metric space $(X, \mathfrak{p})$ converges to a point $x_0 \in X$ if $\lim_{n \rightarrow \infty}\mathfrak{p}(x_n, x_0) = \mathfrak{p}(x_0, x_0)$. A sequence $\{x_n\}$ of elements of $X$ is called \emph{Cauchy} if the limit $\lim_{m, n \rightarrow \infty} \mathfrak{p}(x_n, x_m)$ exists and is finite. The partial metric space $(X, \mathfrak{p})$ is called complete if for each Cauchy sequence $\{x_n\}$, there is some $x \in X$ such that
$$ \lim_{n \rightarrow \infty}\mathfrak{p}(x_n, x) = \mathfrak{p}(x, x) = \lim_{m, n \rightarrow \infty}\mathfrak{p}(x_n, x_m).$$
An example of a partial metric space is the pair $(\mathbb{R}^{+}, \mathfrak{p})$, where $\mathfrak{p}(x, y) = \max\{x, y\}$ for all $x, y \in \mathbb{R}^{+}$. For more material about the partial metric spaces, see, e.g. \cite{A, B, K} and references therein.

In 2012, A. Amini-Harandi \cite{AH} introduced a new extension of the concept of partial metric space, called a \emph{metric-like space}. After that, the concept of $b$-metric-like space which generalizes the notions of partial metric space, metric-like space, and $b$-metric space was introduced by Alghamdi et al. in \cite{Al}. Recently, Zidan and Mostefaoui \cite{Z} introduced the double controlled quasi metric-like spaces and studied some topological properties of this space. Here, we state the concept of a metric-like space.
\begin{definition} A mapping $\mathfrak{D}:X \times X \rightarrow \mathbb{R}^{+}$, where $X$ is a non-empty set, is said to be a metric-like on $X$ if for any $x, y, z \in X$, the following three conditions hold true:\\
(i) $\mathfrak{D}(x, y) = 0 \Rightarrow x = y$;\\
(ii) $\mathfrak{D}(x, y) = \mathfrak{D}(y, x)$;\\
(iii) $\mathfrak{D}(x, y) \leq \mathfrak{D}(x, z) + \mathfrak{D}(z, y)$.
\end{definition}
The pair $(X, \mathfrak{D})$ is called a metric-like space. A metric-like on $X$ satisfies all of the conditions of a metric except that $\mathfrak{D}(x, x)$ may be positive for some $x \in X$. The study of partial metric spaces has wide area of application, especially in computer sciences, see, e.g. \cite{B, Ma, R} and references therein. That is why working on this topic can be very useful in practice. Since metric-likes are generalizations of partial metrics, knowing them can therefore provide us more applicable fields. In fact, this is our motivation to study the metric-like spaces. Each metric-like $\mathfrak{D}$ on $X$ generates a topology $\tau_\mathfrak{D}$ on $X$ whose base is the family of open balls. An open ball in a metric-like space $(X, \mathfrak{D})$, with center $x$ and radius $r > 0$, is the set
$$B(x, r) = \{y \in X : |\mathfrak{D}(x, y) - \mathfrak{D}(x, x)| < r \}.$$
It is clear that a sequence $\{x_n\}$ in the metric-like space $(X, \mathfrak{D})$ converges to a point $x \in X$ if and only if $\lim_{n \rightarrow \infty}\mathfrak{D}(x_n, x) = \mathfrak{D}(x, x)$. A sequence $\{x_n\}$ of elements of a metric-like space $(X, \mathfrak{D})$ is called Cauchy if the limit $\lim_{n, m \rightarrow \infty}\mathfrak{D}(x_n, x_m)$ exists and is finite. The metric-like space $(X, \mathfrak{D})$ is called complete if for each Cauchy sequence $\{x_n\}$, there is some $x_0 \in X$ such that
$$\lim_{n \rightarrow \infty}\mathfrak{D}(x_n, x_0) = \mathfrak{D}(x_0, x_0) = \lim_{n, m \rightarrow \infty}\mathfrak{D}(x_n, x_m).$$
For more details on this topic, see, e.g.  \cite{AH, Ha}. Note that every partial metric space is a metric-like space. But, the converse is  not true in general. For example, let $X = \mathbb{R}$, and let $\mathfrak{D}(x, y) = \max\{|x - 5|, |y - 5|\}$ for all $x, y \in \mathbb{R}$. Then $(X, \mathfrak{D})$ is a metric-like space, but since $\mathfrak{D}(0, 0) \nleq \mathfrak{D}(1, 2)$, then $(X, \mathfrak{D})$ is not a partial metric space. We now state another extension of the notion of distance that allows distance to be a complex value. Azam et al., \cite{Az} introduced the concept of a complex valued metric space and obtained sufficient conditions for the existence of common fixed points of a pair of mappings satisfying contractive type conditions. In that article, they consider a partial order $\precsim$ on the set of complex numbers $\mathbb{C}$ and then introduce a complex valued metric. The partial order $\precsim$ is as follows:
\begin{align*}
z_1 \precsim z_2 \Leftrightarrow \ Re(z_1) \leq Re(z_2), \ Im(z_1) \leq Im(z_2).
\end{align*}
Let $X$ be a non-empty set. Suppose that the mapping $d :X \times X \rightarrow \mathbb{C}$ satisfies the following conditions:\\
(i) $0 \precsim d(x, y)$ for all $x, y \in X$;\\
(ii) $d(x,y) = 0 \Leftrightarrow x = y$;\\
(iii) $d(x,y) = d(y,x)$ for all $x, y \in X$;\\
(iv) $d(x,y) \precsim d(x,z) + d(z, y)$ for all $x, y, z \in X$.\\
Then $d$ is called a complex valued metric on $X$ , and $(X , d)$ is called a complex valued metric space. Combining the two concepts complex valued metric spaces and metric-like spaces, we get complex valued metric-like spaces. Also, we introduce the notion of a complex valued partial metric space. As will be seen, the notion of complex valued metric-like space is a generalization of the notions of metric-like space, complex valued metric space, partial metric space, complex valued partial metric space and metric space. Therefore, it is interesting to investigate this general notion.

In this article, we focus on the structure of complex valued metric-like spaces and study some topological properties
of this space. For instance, we introduce some concepts such as quasi-equal points, completely separate points, convergence of a sequence, Cauchy sequence, cluster point, limit point, complex absolute value and complex diameter of a subset of a complex valued metric-like space.
Additionally, we present several results about complex valued metric-like spaces.

\section{Results and proofs}
Let $\mathbb{C}$ be the set of complex numbers and $z_1, z_2 \in \mathbb{C}$. Following \cite{Az}, we define a partial order $\precsim$ on $\mathbb{C}$ as follows:
\begin{align*}
z_1 \precsim z_2 \Leftrightarrow Re(z_1) \leq Re(z_2), \ Im(z_1) \leq Im(z_2).
\end{align*}
Hence, $z_1 \precsim z_2$ if one of the following conditions is satisfied:\\
(i) $Re(z_1) = Re(z_2)$ and  $Im(z_1) < Im(z_2)$;\\
(ii) $Re(z_1) < Re(z_2)$ and $\ Im(z_1) = Im(z_2)$;\\
(iii) $Re(z_1) < Re(z_2)$ and $\ Im(z_1) < Im(z_2)$;\\
(iv) $Re(z_1) = Re(z_2)$ and $\ Im(z_1) = Im(z_2)$.\\

We write $z_1 \precnsim z_2$ if $z_1 \neq z_2$ and one of (i), (ii), and (iii) is satisfied. Also, we write $z_1 \prec z_2$ if only (iii) is satisfied. Note that

\begin{align*}
& 0 \precsim z_1 \precnsim z_2 \Rightarrow |z_1| < |z_2|, \\ & z_1 \precsim z_2 \prec z_3 \Rightarrow z_1 \prec z_3.
\end{align*}

Also, $z_2 \succsim z_1$ (resp. $z_2 \succ z_1$) means that $z_1 \precsim z_2$ (resp. $z_1 \prec z_2$). Throughout the paper, the set $\{z \in \mathbb{C} \ | \ z \succsim 0 \}$ is denoted by $\mathbb{C} \succsim 0$, i.e. $\mathbb{C} \succsim 0 = \{z \in \mathbb{C} \ | \ z \succsim 0 \}$. A complex number $z$ is called positive if $0 \prec z$.


\begin{definition}
Let $X$ be a non-empty set. A mapping $d: X \times X \rightarrow  \mathbb{C}\succsim 0$ is called a
complex valued metric-like on $X$ if for any $x, y,  z\in X$, the following conditions hold:\\
(D1) $d(x, y)=0 \Rightarrow x=y$;\\
(D2) $d(x, y)=d(y, x)$;\\
(D3)$d(x, y)\precsim d(x, z)+d(z,y).$
\end{definition}

The pair $(X, d)$ is then called a complex valued metric-like space. Indeed, a complex valued metric-like on $X$ satisfies all of the conditions of complex valued metric except that may be $ 0 \precnsim d(x, x)  $ for some $x \in X$. For convenience, we write (CVML) for "complex valued metric-like".

In the following, we provide an example of a CVML space. 

\begin{example}
Let $ X=\mathbb{C}$ . A mapping
$d: \mathbb{C}\times \mathbb{C} \rightarrow\mathbb{C}$ defined by
$d\left(z_{1}, z_{2}\right)=e^{i \theta}\left(| z_{1} |+\left|z_{2}\right|\right)$, where $0 \leqslant \theta \leqslant \frac{\pi}{2}$
is a CVML on $\mathbb{C}$.
 \end{example}

In the following, we introduce the complex absolute value of $z \in \mathbb{C}$ which is denoted by $|\cdot|_{c}$.
$$|z|_{c}=\left|\operatorname{Re(z)}\right|+i\left|Im(z)\right| $$ for any $z \in \mathbb{C}$.
Clearly, $0\precsim |z|_c$ for all $z \in \mathbb{C}$.

\begin{definition}
 Let $(X, d)$ be a CVML space and $A\subseteq X$. An open ball with center $x_0 \in X$ and radius $0 \prec r \in \mathbb{C}$ is the set
\begin{align*}
N(x_0, r)=\left\{y \in X :|d(x_0, y)-d(x_0, x_0)|_{c} \prec r \right\}.
\end{align*}
\end{definition}

\begin{definition} Let $X$ be a non-empty set. A mapping $p: X \times X\rightarrow\mathbb{C} \succsim 0$ is said to be a
complex valued partial metric on $X$ if for all $x, y, z \in X,$ the following conditions hold:
\begin{itemize}
  \item $x=y \Leftrightarrow p(x, x) = p(y, y) = p(x, y)$;
 \item $\max \{p(x, x), p(y, y)\} \precsim p(x, y)$;
  \item $ p(x, y)=p(y, x)$;
   \item $p(x, z)\precsim p(x, y)+p(y, z)-p(y, y)$
  \end{itemize}
    The pair $(X, p)$  is called a complex valued partial metric space.
\end{definition}

 \begin{definition}
Let $(X, d)$ be a CVML space, let $\{x_{n}\}_{n\geq 1}$ be a sequence of $X$ and let $x_{0} \in X$.
We say that
\begin{itemize}
\item The sequence $\{x_{n}\}_{n \geqslant 1}$ converges to $x_{0}$ if for every $0 \prec r \in \mathbb{C} $ there exists $ n_{0} \in \mathbb{N}$ such that $x_n \in N(x_{0}, r )$ for all $n>n_0$, i.e.
$|d(x_{n}, x_{0})-d(x_{0}, x_{0})|_{c} \prec r$ for all $n>n_0$.
We denote this by $\lim_{n \rightarrow +\infty} x_{n}=x_{0},$ or $x_{n} \rightarrow x_{0}$
 as $n \rightarrow +\infty$.
\item The sequence $\{x_{n}\}_{n \geqslant 1}$ is called Cauchy  if $\lim_{m,n \rightarrow +\infty} d(x_{n}, x_{m})$ exists and is finite.
It means that sequence $\{x_{n}\}_{n \geqslant 1}$ is Cauchy if and only if
$\lim_{m,n \rightarrow +\infty} d(x_{n}, x_{m})=z_0$ for some $z_0\in \mathbb{C}$.
 \item The complex valued metric-like space $(X, d)$ is complete if every Cauchy sequence of $X$ is convergent. On the other hand, the complex valued metric-like space $(X, d)$ is complete if for each Cauchy sequence $\{x_n\}$, there is some $x \in X$ such that \begin{align*} \lim_{n \rightarrow +\infty} d(x_{n}, x) = d(x, x) = \lim_{m, n \rightarrow +\infty}d(x_n, x_m)\end{align*}
\end{itemize}
 \end{definition}



\begin{remark} Note that in CVML spaces the limit of a convergent sequence is not necessarily unique and this means that topology of these spaces is not necessarily a Hausdorff topology. For instance, suppose that $X = \mathbb{C}$ and $d(x, y) =  i \max\{|x|, |y|\}$ for each $x, y \in X$. Putting $x_n = \frac{i}{n}$, we have $\lim_{n \rightarrow \infty}d(\frac{i}{n}, i) = \lim_{n \rightarrow \infty} i \max\{|\frac{i}{n}|, |i|\} = i = d(i, i)$. It means that the sequence $\{\frac{i}{n}\}$ converses to $i$, i.e. $\frac{i}{n} \rightarrow i$. Moreover, we have $\lim_{n \rightarrow \infty}d(\frac{i}{n}, 2i) = \lim_{n \rightarrow \infty} i \max\{|\frac{i}{n}|, |2i| \} = 2i = d(2i, 2i)$, and consequently, $\frac{i}{n} \rightarrow 2i$ as well. This demonstrates that the sequence $\{\frac{i}{n}\}$ converges to two different points.
\end{remark}

The above example leads us to the next definition.

\begin{definition} (Quasi-equal points) Let $(X, d)$ be a CVML space. The points $x, y \in X$ are called \emph{quasi-equal points} if there exists a sequence $\{x_n\}$ of $X$ converging to both $x$ and $y$, i.e. $x_n \rightarrow x$ and $x_n \rightarrow y$.
\end{definition}
From the previous remark, one can easily deduce that if $X = \mathbb{C}$ and $d(x, y ) = i \max\{|x|, |y|\}$ for each $x, y \in \mathbb{C}$, then $i$ and $2i$ are two quasi-equal points.

\begin{definition} (Completely separate points) Let $(X, d)$ be a CVML space. The points $x, y$ of $X$ are called \emph{completely separate points} if the following condition holds true:
$$d(x, x) + d(y,y) \prec d(x, y) $$
\end{definition}

\begin{theorem}
Let $(X, d)$ be a CVML space. Then there are no convergent sequences to two completely separate points.
\end{theorem}
\begin{proof} Suppose that $x, y$ are two completely separate points. To obtain a contradiction, let $\{x_n\}$ be a sequence of $X$ converging to both $x, y$. Put $r_1 = \frac{1}{3}r$, where $r = d(x, y) - d(x, x) - d(y, y) \succ 0$. Since $x_n \rightarrow x$ and $x_n \rightarrow y$, for $\varepsilon = r_1$, there exist two positive integers $N_1$ and $N_2$ such that $$\max\{|d(x_N, x) - d(x, x)|, |d(x_N, y) - d(y, y)|\} \prec r_1$$
for all $n \geq N = \max\{N_1, N_2\}$. Therefore,
\begin{align*}
d(x, y) = |d(x, y)|_{c} & \precsim |d(x, x_N) + d(x_N, y)|_{c} \\ & = |d(x, x_N) + d(x_N, y) - d(x, x) + d(x, x) - d(y, y) + d(y, y)|_{c} \\ & \precsim |d(x, x_N) - d(x, x)|_{c} + |d(x_N, y) - d(y, y)|_{c} + d(x, x) + d(y, y) \\ & \prec r_1 + r_1 + d(x, x) + d(y,y) \\ & = 2 r_1 + d(x,x) + d(y,y) \\ & = \frac{2}{3}\Big[d(x, y) - d(x, x) - d(y, y)\Big] + d(x, x) + d(y, y) \\ & = \frac{2}{3}d(x, y) + \frac{1}{3}d(x, x) + \frac{1}{3}d(y, y).
\end{align*}
So we obtain that $d(x, y) \prec d(x, x) + d(y, y)$, a contradiction. This contradiction proves our claim.
\end{proof}
Immediate conclusion from the above theorem demonstrates that completely separate points are not quasi-equal.

\begin{definition} (Cluster points) Let $(X, d)$ be a CVML space and let $A$ be a subset of $X$. A point $x_0 \in X$ is said to be a cluster point of $A$ if for every positive complex number $r$ there exists an element $a \in A$ such that $|d(a, x_0) - d(x_0, x_0)|_{c} \prec r$.
\end{definition}
As usual, the set of all \emph{cluster points} of $A$ is called the closure of $A$ and is denoted by $\overline{A}$. Note that $x_0 \in \overline{A}$ if and only if $N(x_0, r) \cap A \neq \phi$ for all $0 \prec r \in \mathbb{C}$. Indeed, we have
\begin{align*}
\overline{A} = \Big\{x_0 \in X : N(x_0, r) \cap A \neq \phi \ \ for \ all \ 0 \prec r \in \mathbb{C}\Big\}
\end{align*}

It is clear that if $(X, d)$ is a CVML space and $A$ is a subset of $X$, then $A \subseteq \overline{A}$. In the following, we establish a theorem to present a necessary and sufficient condition for cluster points in the CVML spaces. First, we prove a theorem about the convergence of sequences in CVML spaces.

\begin{theorem}
\label{1}
Let $(X, d)$ be a CVML space and let $\left\{x_{n}\right\}$ be a sequence of $X$. Then $\left\{x_{n}\right\}$ converges to $x$
if and only if $\lim_{n \rightarrow+\infty}\Big|d\left(x_{n}, x\right)-d(x, x)\Big| =0$.
 \end{theorem}

\begin{proof}
Suppose that $\left\{x_{n}\right\}$ converges to $x$. Let $\varepsilon = \frac{\alpha}{\sqrt{2}}+i \frac{\alpha}{\sqrt{2}}=\frac{\alpha}{\sqrt{2}}(1+i)$, where
$\alpha$ is a given positive real number. Since  $\left\{x_{n}\right\}$
converges to $x$, there exists a positive integer $N$
such that
$\left|d\left(x_{n}, x\right)-d(x, x)\right|_{c} \prec \varepsilon $ for all $n>N$.
Therefore,
$$\Big| | d\left(x_{n}, x\right)- d(x, x) |_{c}\Big| < |\varepsilon| = \alpha$$
for all $n>N$,
which means that
$\left|d\left(x_{n}, x\right)-d(x, x)\right| < \alpha$ for all $n > N$.
So
$$ \lim_{n \rightarrow+\infty} \left|d\left(x_{n}, x\right)-d(x, x)\right|=0.$$

Conversely, suppose that $\lim_{n \rightarrow+\infty} |d\left(x_{n}, x\right)-d(x, x) | = 0$. Therefore, for each positive real number $\alpha$, there exists a positive integer $N$ such that $\left|d(x_{n}, x)-d(x, x)\right| < \alpha$ for all $n > N$.
It is clear that for any $ 0 \prec \varepsilon \in \mathbb{C}$, there exists a positive real number $\alpha$ such that
$ \alpha+ i\alpha=(1+i) \alpha \prec \varepsilon.$
Since $\lim_{n \rightarrow+\infty}  \left|d\left(x_{n}, x\right)-d(x, x)\right|=0,$ there
exists a positive integer $n_{0} \in \mathbb{N}$ such  that
$\left|d\left(x_{n}, x\right)-d(x, x)\right|<\alpha$ for all $ n \geq n_{0}$.
Hence, for all $n \geq n_{0}$, we have

\begin{eqnarray*} \left| d\left(x_{n}, x\right)- d(x, x)\right|_{c}  \precsim\left(1+i \right) | d\left(x_{n}, x\right)-d\left(x , x\right)|   \prec (1+i) \alpha \prec \varepsilon \end{eqnarray*}
Indeed, we get that for all $0 \prec \varepsilon \in \mathbb{C}$ there exists a positive integer $n_0$ such that $|d \left(x_{n}, x\right)-\left.d(x, n)\right|_{c} \prec \varepsilon$ for all $n \geq n_0$, and
this means that
$
\lim_{n \rightarrow+\infty}x_{n} = x $, as desired.
 \end{proof}

\begin{theorem} \label{2*}Let $(X, d)$ be a CVML space and let $A$ be a subset of $X$. Then $x_0 \in \overline{\mathcal{A}}$ if and only if there is a sequence $\{a_n\} \subseteq \mathcal{A}$ converging to $x_0$.
\end{theorem}

\begin{proof} Suppose that $x_0 \in \overline{A}$. So for each $r_n = \frac{1+i}{n}$ ($n \in \mathbb{N}$), there is an element $a_n \in \mathcal{A}$ such that $|d(a_n, x_0) - d(x_0, x_0)|_{c} \prec \frac{1+i}{n}$. It implies that $\lim_{n \rightarrow +\infty}|d(a_n, x_0) - d(x_0, x_0)|= 0$, and it follows from Theorem \ref{1} that the sequence $\{a_n\} \subseteq A$ converges to $x_0$. Conversely, assume that $\{a_n\}$ is a sequence of $A$ converging to $x_0$. We must to show that $x_0 \in \overline{A}$. Let $r$ be an arbitrary positive complex number. Therefore, there exists a positive number $N$ such that $a_n \in N(x_0, r)$ for all $n \geq N$. It means that $A \cap N(x_0, r) \neq \phi$. Since we are assuming that $r$ is an arbitrary positive complex number, it is deduced that $x_0 \in \overline{A}$, as desired.
\end{proof}

\begin{definition} (Limit points) Let $(X, d)$ be a CVML space and let $A$ be a subset of $X$. A point $x_0 \in X$ is said to be a limit point of $A$ if  $N(x_0,r) \cap (A - \{x_0\}) \neq \phi$ for all positive complex numbers $r$.
\end{definition}
As usual, the set of all \emph{limit points} of $A$ is denoted by $A'$. A subset $A$ of $X$ is called a closed set, whenever each limit point of $A$ belongs to itself, i.e. $A' \subseteq A.$ One can easily prove the following theorem.

\begin{theorem} Let $(X, d)$ be a CVML space and let $A$ be a subset of $X$. Then $\overline{A} = A \cup A'$.
\end{theorem}

\begin{definition} Let $(X, d)$ be a CVML space and let $A$ be a subset of $X$. The complex diameter of $A$ is defined as follows:
$$ diam_c(A) : = \sup \Big\{|d(x, y) - d(x, x)|_{c}, |d(x, y) - d(y, y)|_{c} : x, y \in A \Big\}.$$
If $|diam_c(A)| < \infty$, then the subset $A \subseteq X$ is said to be bounded.
\end{definition}

\begin{example} Let $X = \mathbb{C}$, $d(x, y) = (1+i)(|z_1| + |z_2|)$, and let $A = \{z \in \mathbb{C} : 3 < |z| < 5\}$. In this case, we have
\begin{align*}
diam_c(A) & = \sup\Big\{|d(z_1, z_2) - d(z_1, z_1)|_{c}, |d(z_1, z_2) - d(z_2,z_2)|_{c} : z_1, z_2 \in A \Big\} \\ & = \sup\Big\{\big||z_1| - |z_2|\big| + i \big||z_1| - |z_2|\big| : z_1, z_2 \in A \Big\} \\ & = \sup\Big\{\big||z_1| - |z_2|\big|(1+i) : z_1, z_2 \in A \Big\} \\ & = 2(1+i).
\end{align*}
\end{example}

\vspace{.25cm}




\bibliographystyle{amsplain}

\begin{thebibliography}{99}


\bibitem{AH}  A. Amini-Harandi, Metric-like spaces, partial metric spaces and fixed points, Fixed Point Theory Appl., \textbf{204} (2012), 2-10.

\bibitem{Az}A. Azam, B. Fisher and M. Khan, Common Fixed Point Theorems in Complex Valued Metric Spaces, Numer. Funct. Anal. Optim. \textbf{32}(3) (2011), 243–-253.

\bibitem{A}  I. Altun, F. Sola, and H. Simsek, Generalized contractions on partial metric spaces, Topol. Appl., \textbf{157} (2010),  2778-2785.

\bibitem{Al}  M. A. Alghamdi, N. Hussain and P. Salimi, Fixed point and coupled fixed point theorems
on b-metric-like spaces, J Inequal Appl, \textbf{402} (2013). https://doi.org/10.1186/1029-242X-2013-402.





\bibitem{B} M. Bukatin, R. Kopperman, S. G. Matthews, and H. Pajoohesh, Partial metric spaces, Am. Math. Mon., \textbf{116} (2009), 708-718.


\bibitem{De} M. M. Deza and E. Deza, Encyclopedia of distances, Springer, Berlin, 2009.



\bibitem{Ha} B. Hazarika, E. Karapinar, R. Arab and M. abbani, Metric-like spaces to prove existence of solution for nonlinear
quadratic integral equation and numerical method to solve it, Journal of Computational and Applied Mathematics, \textbf{328} (2018), 302-313.



\bibitem{H} Amin Hosseini and Ajda Fon$\check{s}$er, On the Structure of Metric-like Spaces, Sahand Communications in Mathematical Analysis (SCMA), \textbf{14}(1) (2019), 159-171.


\bibitem{K} E. Karapinar and IM. Erha, Fixed point theorems for operators on partial metric spaces, Appl. Math. Lett., \textbf{24} (2011), 1894-1899


\bibitem{Ma} S. G. Matthews, Metric domains for completeness, PhD thesis, University of Warwick, Academic
Press, 1986.

\bibitem{Mat} S. G. Matthews, Partial metric topology, Proc. 8th Summer Conference on General Topology and Applications, Annals of the New York Academy of Sciences, \textbf{728} (1994), 183-197.




\bibitem{O} S. J. O'Neill, Partial metrics, valuations and domain theory, Proc. 11th Summer Conference
on General Topology and Applications, Annals of the New York Academy of Sciences, \textbf{806} (1996), 304–315.



\bibitem{R} J. J. M. M. Rutten, Elements of generalized ultrametric domain theory, Theor. Comput. Sci., \textbf{170}
(1996), 349-381.



\bibitem{Z} A. M. Zidan and Z. Mostefaoui, Double controlled quasi metric-like spaces and some topological properties
of this space, AIMS Mathematics, \textbf{6}(10) (2021), 11584-11594



\end{thebibliography}

\end{document}